\documentclass[12pt]{amsart}

\usepackage[left=1.02in,top=1.02in,right=1.02in, bottom=1.02in]{geometry}
\usepackage{amsmath,amsthm,amssymb,graphicx,color}
\usepackage{epsfig,amsfonts,latexsym}
\usepackage{enumerate}

\usepackage{bm}
\newcommand*{\B}[1]{\ifmmode\bm{#1}\else\textbf{#1}\fi}
\usepackage[round]{natbib}
\usepackage{hyperref}
\hypersetup{colorlinks,
citecolor=blue,
linkcolor=blue,
urlcolor=black
}

\numberwithin{equation}{section}

\newcommand{\tr}{\mbox{tr}}

 \newcommand{\VV}{\mathbb{V}}

 \newcommand{\UU}{\mathbb{U}}
 \newcommand{\HH}{\mathbb{H}}
 \newcommand{\NN}{\mathbb{N}}
 
 \newcommand{\RR}{\mathbb{R}}

 \newcommand{\PP}{\mathbb{P}}

 \newcommand{\HS}{\rm {HS\ }}

 \newcommand{\Hess}{\mathbf {Hess\ }}

\newtheorem{theorem}{Theorem}[section]
\newtheorem{thm}[theorem]{Theorem}
\newtheorem{lemma}[theorem]{Lemma}

\newtheorem{proposition}[theorem]{Proposition}

\theoremstyle{definition}
\newtheorem{definition}[theorem]{Definition}

\theoremstyle{remark}
\newtheorem{remark}[theorem]{Remark}
\newtheorem{rem}[theorem]{Remark}

\newcount\Comments  
\definecolor{darkgreen}{rgb}{0,0.5,0}
\definecolor{purple}{rgb}{1,0,1}
\newcommand{\kibitz}[2]{\ifnum\Comments=1\textcolor{#1}{#2}\fi}

\allowdisplaybreaks

\begin{document}

\author{Shulan Hu$^1$, Ran Wang$^{2*}$}
 \address{$^1$  School of Statistics and Mathematics, Zhongnan University of Economics and Law, Wuhan, 430073, China.} \email{hu\_shulan@zuel.edu.cn}
\address{$^{2*}$  Corresponding author, School of Mathematics and Statistics, Wuhan University, Wuhan, 430072, China. }
 \email{rwang@whu.edu.cn}

   \title{Asymptotics  of stochastic Burgers equation   with jumps}

\date{}
\begin{abstract} For one-dimensional stochastic Burgers equation driven by Brownian motion and Poisson process, we study the  $\psi$-uniformly exponential ergodicity with $\psi(x)=1+\|x\|$, the moderate deviation principle  and the large deviation principle for the occupation measures.
\end{abstract}



\maketitle

\section{Introduction}
 
As is well-known, Burgers equation was first studied to understand the turbulent fluid flow, see \cite{Burgers}.   Since then the Burgers equation perturbed by different random noises have been considered,  see monographs  \cite{DPZ14}, \cite{PeZa07} and recent articles  \cite{DX2007}, \cite{WuXie}, \cite{DXZ} and references therein.

The ergodicity of the  stochastic Burgers equation driven by Brownian motion and Poisson process  was proved  in \cite{Dong2008}  in the sense that the  system converges to a unique invariant measure
under the weak topology,  but the convergence speed is not addressed.
In this paper,
we prove that the system  converges to the invariant measure  exponentially faster under a
topology stronger than total variation by  constructing a Lyapunov function in the same  way as in   \cite{DXW}.    The
moderate deviation principle (MDP) for the occupation  measure   is also obtained.

 The large deviation principle (LDP) for the occupation  measure  is one of the strongest ergodicity results
for the long time behavior of Markov processes. It has been one of the classical research
topics in probability since the pioneering work of \cite{DV}. It gives an estimate on the probability that the occupation measures are deviated from the invariant measure,  refer to  \cite{DS} for  an introduction to large deviation theory of Markov processes.
\cite{Wu01} gave   the hyper-exponential recurrence criterion of the  LDP of occupation measures for  strong Feller and irreducible Markov processes.
 Based on this   criterion,   the  large deviations  of the occupation measures for the stochastic  Burgers equation  and stochastic Navier-Stokes equation   driven by Brownian motion are proved in \cite{Gou1} and \cite{Gou2}.
 There are some other papers about the applications of Wu's criterion, see \cite{JNPS2}, \cite{Ner} for some dissipative SPDEs.   \cite{WXX} proposed a  framework  for verifying the hyper-exponential recurrence condition,  which contains a   family of  strong dissipative SPDEs. In that framework, the strong dissipation produces to a  stronger-norm moment estimate for the system after a fixed time  uniformly over the initial values, which implies the  hyper-exponential recurrence condition.    See \cite{WX} for an application to  stochastic reaction-diffusion equation driven by the subordinate Brownian motion.

  However, the framework in \cite{WXX} is no longer available for  the stochastic Burgers equation, which does not have the strong dissipation.  In this paper, we check the hyper-exponential recurrence condition   by using an exponential martingale argument.  Due to the present of the jumps, the proof here is more complicated than that for the Brownian motion case in \cite{Gou1}.

The paper is organized as follows.  The framework is given in Section 2. Section 3 is devoted to proving    the  $\psi$-uniformly exponential ergodicity and the moderate deviation principle.   In Section 4, we   prove the large deviation principle.

\section{The framework}

  Let $\HH:=L^2(0,1)$ with the Dirichlet boundary condition and with vanishing mean values.
Then $\HH$ is a real separable Hilbert space with inner product
$$\langle  x,y \rangle:=\int_0^1x(\xi)y(\xi) d\xi,\ \ \ \ \ \forall \ x, y \in \HH.$$
Denote $\|x\|_{\HH}:= \left(\langle x,x\rangle_{\HH}\right)^{\frac12}.$
Let $\Delta x=x''$ be the second order differential operator on $\HH$. Then $-\Delta$ is a positive  self-adjoint operator on $\HH$. Let  $\alpha_{k}=\pi^2k^2$  and  $e_{k}(\xi):= \sqrt{2}\sin(k\pi \xi)$, for any $k\in \mathbb N^*=\{1,2,\cdots\}$.
Then $\{e_k\}_{k\in\NN^*}$ forms an orthogonal basis of $\HH$ and  $-\Delta e_k=\alpha_k e_k$ for any $k\in\NN^*$.

  Let $ \VV$ be the domain of the fractional operator $(-\Delta)^{\frac{1}{2}}$, i.e.,
 $$
 \VV:=\left\{\sum_{k\in \NN^*}\alpha_k^{\frac{1}{2}}a_k e_k; (a_k)_{k\in \NN}\subset\mathbb R \text{ with } \sum_{k\in \NN^*}a_k^2<+\infty \right\},
  $$
 with the inner product
 $$
 \langle x,y\rangle_{\VV}:=\sum_{k\in \NN^*} \alpha_k\langle x, e_k\rangle_{\HH}\cdot\langle y, e_k\rangle_{\HH},
 $$
 and with the norm
$\|x\|_{\VV}:=\langle x,x\rangle_{\VV}^{\frac12}.
 $
Clearly, $\VV$ is densely and compactly embedded in $\HH$.

Let $(\Omega, \mathcal F, \{\mathcal F_t\}_{t\ge0}, \mathbb P)$ be a completed filtered probability space, and 
$N(dt,du)$  the Poisson measure with finite intensity measure $n(du)$ on a given
measurable space $(\mathbb U, \mathcal B(\mathbb U))$. Then
 $$ \widetilde N(dt,du):= N(dt,du)-n(du)dt$$
is the compensated  martingale measure. Let $W$ be the  cylindrical Wiener process, which is independent with $N(dt,du)$, e.g.,
  $
  W:=\sum_{k\in \NN^*}W^k e_k, 
  $
  where $\{W^k\}_{k\in\NN^*}$ are a sequence of independent standard one-dimensional Brownian motions   independent with $N(dt,du)$.  

Consider the following stochastic Burgers equation in the Hilbert space $\HH$:
\begin{eqnarray}\label{eq:Burgers}
\left\{
 \begin{array}{lll}
 & dX_t
 =  \Delta X_tdt+B(X_t)dt+ QdW_t+\int_{\mathbb U} f(X_{t-},u)\widetilde{N}(dt,du), \\
    &X(0)=x\in \HH
 \end{array}
\right.
\end{eqnarray}
Here $B(x):=B(x,x)$ is a bilinear operator, which is defined by  $B(x,y):=xy'$ for $x\in \HH, y\in \VV$, and $Q\in \mathcal L(\HH)$ (the   space of all Hilbert-Schmidt operators from $\HH$ to $\HH$) is given by
$$
Qx=\sum_{k\in \NN^*}\beta_k \langle x, e_k\rangle e_k, \ \ x\in \HH,
$$
with $\|Q\|_{\HS}:=\sqrt{\tr(Q^*Q)}=\sqrt{\sum_{k\in \NN^*} |\beta_k|^2}<\infty$.

Assume that the coefficient $f$ satisfies the following conditions:
\begin{itemize}
  \item[(H.1)] $f(\cdot,\cdot):\HH\times \UU\rightarrow \HH$ is measurable;
  \item[(H.2)] $\int_{\UU}\|f(0, u)\|_{\HH}^2 n(du)<\infty$;
  \item[(H.3)] $\int_{\UU}\|f(x, u)-f(y,u)\|_{\HH}^2 n(du)\le K\|x-y\|_{\HH}^2, \ \forall x,y\in \HH$;
  \item[(H.4)] $f(\cdot, u)\in C_b^1(\HH), \forall u\in \UU$.
\end{itemize}

  Let $\mathbb D([0,+\infty); \HH)$ be the space of all c\`adl\`ag functions from $[0,+\infty)$ to $\HH$ equipped with the  Skorokhod topology.
Denote by $S(t)=e^{\Delta t}$.    
\begin{definition}  The process $X=\{X_t\}_{t\ge0}$ is called a mild solution of \eqref{eq:Burgers}, if for any $x\in \HH$,  $X\in\mathbb D([0,+\infty); \HH)$ satisfying that  for  any  $t>0$,
$$
\int_0^t\left[\|X(s)\|_{\HH}^2+\|B(X(s))\|_{\HH}^2\right]ds<\infty, 
$$
and
\begin{align*}
X_t=&S(t)x+\int_0^t S(t-s)B(X_s)ds+\int_0^t S(t-s)Q dW_s\\
&+\int_0^t\int_{\mathbb U}S(t-s)f(X_{s-},u)\widetilde N(ds,du), \ \ \mathbb P-a.s.
\end{align*}
\end{definition}

 For all $\varphi\in \mathcal B_b(\HH)$ (the space of all bounded measurable functions on $\HH$), define
$$
P_t\varphi(x):=\mathbb E_x[\varphi(X_t)]  \ \ \ \text{for all } t\ge0, x\in \HH,
$$
where $\mathbb E_x$ denotes the expectation with respect to (w.r.t. for short) the law of stochastic process $X$ with initial value $X_0=x$.
For any $t>0$, $P_t$ is said to be {\it strong Feller} if $P_t\varphi\in C_b(\HH)$ for any $\varphi\in \mathcal B_b(\HH)$, where  $C_b(\HH)$ is the space of all bounded continuous functions on $\HH$.
 $P_t$ is {\it irreducible} in $\HH$ if $P_t1_O(x)>0$ for any $x\in \HH$ and any non-empty open subset $O$ of $\HH$.

Recall the following properties about the solution to Eq.\eqref{eq:Burgers}.
\begin{thm}[\cite{DX2007}, \cite{Dong2008}] \label{Thm solu} Under (H.1)-(H.4), the following statements hold:
\begin{enumerate}
\item[(i)]   For every $x \in \HH$ and $\omega \in \Omega$ a.s.,
Eq.\eqref{eq:Burgers} admits a unique mild solution $X=\{X_t\}_{t\ge0}\in \mathbb D([0,\infty);\HH) \cap L^2((0,\infty);\VV)$, which is a Markov process.
\item[(ii)]  $X=\{X_t\}_{t\ge0}$ is   strong Feller and  irreducible in $\HH$, and it admits a unique invariant probability measure $\mu$.
     \end{enumerate}
\end{thm}

Define the occupation measure $\mathcal L_t$   by
\begin{equation}\label{e:occupation}
\mathcal{L}_t(\Gamma):=\frac1t\int_0^t\delta_{X_s}(\Gamma)d s,
\end{equation}
where $\Gamma$ is a Borel measurable set in $\HH$, $\delta_{\cdot}$ is the Dirac measure. Then $\mathcal L_t$ is in  $\mathcal M_1(\HH)$, the space of  probability measures on $\HH$. On  $\mathcal M_1(\HH)$, let  $\sigma(\mathcal M_1(\HH), \mathcal B_b (\HH))$ be the $\tau$-topology  of convergence against measurable and bounded functions,  which is much stronger than the usual weak convergence topology $\sigma(\mathcal M_1(\HH), C_b(\HH))$.

\section{$\psi$-uniformly exponential ergodicity and moderate deviation principle}

Let $\mathcal M_b(\HH)$ be the space of signed $\sigma$-additive measures of bounded variation on $\HH$ equipped with the Borel $\sigma$-field $\mathcal B (\HH)$.
On  $\mathcal M_b(\HH)$, we  consider the $\tau$-topology  $\sigma(\mathcal M_b(\HH), \mathcal B_b(\HH))$.

 Given a measurable function $\psi:\HH\rightarrow\mathbb R_+$, define
 $$
 \mathcal B_{\psi}:=\{g: \HH\rightarrow \mathbb R; |g(x)|\le \psi(x)\}.
  $$
For a  function  $b(t):\mathbb R_+\rightarrow (0,+\infty)$, 
define
\begin{equation}\label{eq: M measure}\mathfrak{M}_t:=\frac{1}{b(t)\sqrt t}\int_0^t(\delta_{X_s}-\mu)d s.\end{equation}
where $b(t)$  satisfies
\begin{equation}\label{eq: scale}
\lim_{t\rightarrow \infty}b(t)=+\infty, \ \ \ \lim_{t\rightarrow \infty}\frac{b(t)}{\sqrt t}=0.
\end{equation}
Let $\PP_{\nu}$ be the probability measure of the system $X$ with initial measure $\nu$.

 \begin{itemize}
  \item[(H.5)]
 Assume that  there exists a constant $M>0$ satisfying that
 \begin{equation}
M:=\sup_{x\in\HH}\int_{\UU}\|f(x, u)\|_{\HH}^2  n(du)<+\infty.
\end{equation}
 \end{itemize}

\begin{thm}\label{thm main 1}
 Assume (H.1)-(H.5) hold. Then the following statements hold for $\psi(x)=1+\|x\|_{\HH}$.
  \begin{itemize}
    \item[(1)] The invariant measure $\mu$ satisfies that $\mu(\psi)<\infty$ and the Markov semigroup $\{P_t\}_{t\ge0}$ is $\psi$-uniformly exponentially ergodic, i.e.,  there exist some constants $C, \gamma>0$ satisfying that for
    $$
    \sup_{g\in \mathcal B_{\psi}}|P_tg(x)-\mu(g)|\le Ce^{-\gamma t}\psi(x), \ \ \ x\in \HH,\ t\ge0.
    $$
    \item[(2)]
  For any initial measure $\nu$ verifying $\nu(\psi)<+\infty$,   the measure $\PP_{\nu}(\mathfrak M_t\in\cdot)$ satisfies the large deviation principle w.r.t. the $\tau$-topology with speed $b^2(t)$ and the rate function
 \begin{equation}
 I(\nu):=\sup\left\{\int \varphi d \nu-\frac12\sigma^2(\varphi);\varphi\in \mathcal B_b(\HH) \right\}, \ \  \ \forall \nu\in \mathcal M_b(\HH),
 \end{equation}
where 
\begin{equation}
\sigma^2(\varphi):=\lim_{t\rightarrow \infty}\frac1t\mathbb E_{\mu}\left(\int_0^t (\varphi(X_s)-\mu(\varphi))d s \right)^2
\end{equation}
exists in $\RR$ for every $\varphi\in  B_{\psi}$.  More precisely, the following three properties hold:
\begin{itemize}
  \item[(a1)] for any $r\ge0$, $\{\beta\in \mathcal M_b(\HH); I(\beta)\le r \}$ is compact in  $(\mathcal M_b(\HH),\tau)$;
  \item[(a2)]$($the upper bound$)$ for any  closed set $\mathcal E$ in $(\mathcal M_b(\HH), \tau)$,
   $$
   \limsup_{t\rightarrow \infty}\frac1{b^2(t)}\log \mathbb P_{\beta}(\mathfrak M_t\in \mathcal E)\le -\inf_{\beta\in \mathcal E} I(\beta);
   $$
    \item[(a3)] $($the lower bound$)$ for any open set $\mathcal D$ in $(\mathcal M_b(\HH), \tau)$,
   $$
   \liminf_{t\rightarrow \infty}\frac1{b^2(t)}\log \mathbb P_{\beta}(\mathfrak M_t\in \mathcal D)\ge -\inf_{\beta\in \mathcal D} I(\beta).
   $$
\end{itemize}
  \end{itemize}

  \end{thm}

Recall that a measurable function $h:\HH\rightarrow \mathbb R$ belongs to the extended domain $\mathbf{D}_e(\mathfrak L)$ of the generator $\mathfrak L$ of $\{P_t\}_{t\ge0}$,
if there is a measurable function $g:\HH\rightarrow\RR$ satisfying  that  for all $t>0$, $\int_0^t|g|(X_s)ds<+\infty, \PP_x$-a.s., and
\begin{equation}
h(X_t)-h(X_0)-\int_0^t g(X_s)ds,
\end{equation}
is a c\`{a}dl\`{a}g $\PP_x$-local martingale for all $x\in \HH$. In that case, we write
$
g:= \mathfrak{L} h.
$

\begin{proof}[Proof of Theorem \ref{thm main 1}]  According to \cite[Theorem 5.2c]{DWT} and \cite[Theorem 2.4]{Wu01},  it is sufficient to prove that there  exist some continuous function $1\le \psi\in \mathbf{D}_e(\mathfrak L)$,  compact subset $\mathcal K\subset \HH$ and constants $\varepsilon, C>0$ such that
\begin{align}\label{Lyapunov}
 -\frac{\mathfrak L \psi}{\psi}\ge\varepsilon \mathrm 1_{\mathcal K^c}-C\mathrm 1_{\mathcal K}.
\end{align}
Here, we construct the Lyapunov   function $\psi$ in the same way as in \cite{DXW}. Since $1+\|x\|_{\HH}$ is comparable with $(1+\|x\|_{\HH}^2)^{\frac12}$, we will take 
\begin{equation}\label{eq: psi}\psi(x)=(1+\|x\|_{\HH}^2)^{\frac12}
\end{equation}
 instead of $1+\|x\|_{\HH}$.
First observe that
\begin{equation}\label{eq: nabla psi}
\nabla \psi(x)= \frac{x}{({1+\|x\|_{\HH}^2})^{\frac12}},
\end{equation}
and
\begin{equation}\label{eq: hess psi}
\Hess \psi(x)=-\frac{ x\times x}{\left(1+\|x\|_{\HH}^2\right)^{\frac32}}+ \frac{I_H}{(1+ \|x\|_{\HH}^2)^{\frac12}},
\end{equation}
here  $I_H$ stands for the identity operator. Then, we have
\begin{equation}\label{eq: nabla hess psi}
\sup_{x\in\HH}\|\Hess \psi(x)\|\le 1, \ \ \ \ \sup_{x\in\HH}\|\nabla \psi(x)\|\le 1,
\end{equation}
here $\|\Hess \psi(x)\|$ and $\|\nabla \psi(x)\|$ denote  their operator norms.   Moreover, we have
\begin{equation}
\langle \Delta x, \nabla \psi(x)\rangle=\frac{ \langle \Delta x, x\rangle}{(1+ \|x\|_{\HH}^2)^{\frac12}}=- \frac{ \|x\|_{\VV}^2}{(1+ \|x\|_{\HH}^2)^{\frac12}} , \  \ \forall x\in \mathbb V,
\end{equation}
and
\begin{equation}\label{eq: B psi}
\langle B(x), \nabla \psi(x)\rangle=\frac{ \langle B(x), x\rangle}{(1+ \|x\|_{\HH}^2)^{\frac12}}=0, \ \  \forall x\in \mathbb V.
\end{equation}

 By Taylor's expansion, for any $x\in \HH, u\in \UU$,  there exists  constant  $\theta\in(0,1)$ satisfying that
\begin{align}\label{eq:psi1}
&\psi\big(x+f(x,u)\big)-\psi(x)-\langle \nabla \psi(x), f(x,u)\rangle \notag \\
= & \frac12 \big\langle \Hess \psi\big(x+\theta f(x,u)\big)f(x,u), f(x,u) \big\rangle.
\end{align}

 By It\^o's formula, we have
\begin{align}\label{eq:psi2}
d\psi(X_t)
=&\langle \Delta X_t, \nabla \psi(X_t)\rangle dt+\langle B(X_t), \nabla\psi(X_t)\rangle dt\notag\\
&+\langle \nabla\psi(X_t), QdW_t\rangle+ \frac12\tr(Q^*{\Hess \psi}(X_t)Q)dt \notag \\
&+\int_{\UU}\big( \psi(X_{t-}+f(X_{t-},u))-\psi(X_{t-}) \big)\widetilde N(dt,du)\notag\\
&+\int_{\UU}\big( \psi(X_{t-}+f(X_{t-},u))-\psi(X_{t-})-\langle \nabla \psi(X_{t-}), f(X_{t-},u) \rangle  \big)n(du)dt.
 \end{align}
Then,  by (H.5), \eqref{eq: nabla psi}-\eqref{eq:psi1}, we know that
\begin{align}\label{eq:psi3}
\mathfrak L\psi(x)=&\langle  \Delta x, \nabla \psi(x)\rangle  +\langle  B(x), \nabla\psi(x)\rangle  +  \frac12\tr(Q^*{\Hess \psi}(x)Q) \notag \\
&+\int_{\UU}\left( \psi(x+f(x,u))-\psi(x)-\langle \nabla \psi(x), f(x,u) \rangle  \right)n(du) \notag\\
\le &-\frac{\|x\|_{\VV}^2}{(1+\|x\|_{\HH}^2)^{\frac12}}+\frac12 \|Q\|_{\HS}^2 +\frac12\int_{\UU} \|f(x,u)\|_{\HH}^2n(du)\notag\\
\le&-\frac{1+\|x\|_{\VV}^2}{(1+\|x\|_{\HH}^2)^{\frac12}}+\frac{1}{(1+\|x\|_{\HH}^2)^{\frac12}}+\frac12 \|Q\|_{\HS}^2+\frac{M}{2}\notag\\
\le &-(1+\|x\|_{\VV}^2)^{\frac12}+c_1,
 \end{align}
 where in the last inequality  the Poincar\'e inequality 
 $\|x\|_{\VV}\ge  \pi \|x\|_{\HH}$ is used,  $c_1:=1+\frac12\left( \|Q\|_{\HS}^2+M\right)$. 
 
 Let $\mathcal K:=\{x\in \HH; \|x\|_{\VV}\le 2c_1\}$. Then $\mathcal K$ is a compact set in $\HH$. For any $x\in \mathcal K$, we have 
  \begin{align}\label{eq:psi52} 
   \frac{(1+\|x\|_{\VV}^2)^{\frac12}-c_1}{(1+\|x\|_{\HH}^2)^{\frac12}} \ge -c_1;
 \end{align}
 for any $x\notin \mathcal K$, we have 
  \begin{align}\label{eq:psi53}  \frac{(1+\|x\|_{\VV}^2)^{\frac12}-c_1}{(1+\|x\|_{\HH}^2)^{\frac12}} \ge  
   \frac{(1+\|x\|_{\VV}^2)^{\frac12}-\frac{\|x\|_{\VV}}{2}}{(1+\|x\|_{\HH}^2)^{\frac12}} \ge \frac12.
   \end{align}
Putting \eqref{eq:psi3}-\eqref{eq:psi53} together,  we obtain that
 \begin{align}\label{eq:psi5} 
-\frac{\mathfrak L\psi(x)}{\psi(x)}\ge \frac{(1+\|x\|_{\VV}^2)^{\frac12}-c_1}{(1+\|x\|_{\HH}^2)^{\frac12}} \ge  &\frac{1}{2}1_{\mathcal K^c}-c_1 1_{\mathcal K},
 \end{align}
 which implies \eqref{Lyapunov}.
 The proof is complete.
\end{proof}

\section{Large deviation principle}

\begin{itemize}
  \item[(H.6)]
  Assume that   there exists a constant $a_0>0$ satisfying that
 \begin{equation}
\sup_{x\in \HH}\int_{\UU}\| f(x, u)\|_{\HH}^2\exp\left(a_0 \|  f(x, u)\|_{\HH} \right)n(du)<+\infty.
\end{equation}

\end{itemize}
For any $\lambda_0>0, L>0$, let
\begin{equation}\label{eq: M}\mathcal M_{\lambda_0,L}:=\left\{\nu\in\mathcal M_1(\HH); \int e^{\lambda_0\|x\|_{\HH}}\nu(dx)\le L\right\}.
\end{equation}

\begin{theorem}\label{thm main}  Assume   (H.1)-(H.4) and (H.6) hold.
Then the family $\PP_{\nu}(\mathcal L_t\in \cdot)$ as $t\rightarrow +\infty$ satisfies the LDP with respect to the $\tau$-topology, with  the speed $t$ and the rate function $J$, uniformly for any initial measure $\nu$ in $\mathcal M_{\lambda_0,L}$. 
More precisely, the following three properties hold:
\begin{itemize}
  \item[(a1)]  for any $a\ge0$, $\{\beta\in\mathcal M_1(\HH); J(\beta)\le a \}$ is compact in  $(\mathcal M_1(\HH),\tau)$;
  \item[(a2)](the upper bound) for any  closed set $\mathcal E$ in  $(\mathcal M_1(E),\tau)$,
   $$
   \limsup_{t\rightarrow \infty}\frac1t\log\sup_{\nu\in \mathcal M_{\lambda_0,L}}\mathbb P_{\nu}(\mathcal L_t\in \mathcal E)\le -\inf_{\beta\in \mathcal E} J(\beta);
   $$
  \item[(a3)](the lower bound) for any  open set $\mathcal D$ in $(\mathcal M_1(E),\tau)$,
   $$
   \liminf_{t\rightarrow \infty}\frac1t\log\inf_{\nu\in \mathcal M_{\lambda_0,L}}\mathbb P_{\nu}(\mathcal L_t\in \mathcal D)\ge -\inf_{\beta\in \mathcal D} J(\beta).
   $$
\end{itemize}
\end{theorem}

\begin{rem} Assumptions   (H.1)-(H.4) are standard conditions for the existence and uniqueness of the solution for  Eq.\eqref{eq:Burgers}, see \cite{DX2007}. Condition (H.6) ((H.5) resp.) guarantees  for the exponential (square resp.) integrability of the solution.  The similar conditions are often used in the study of the large deviation theory for small Poisson noise perturbations of SPDEs, e.g., see \cite[Section 4]{RZ},  \cite[Condition 3.1]{BCD2013} and \cite[Condition 3.1]{YZZ}.  Inspiring by \cite[Section 4.1]{BCD2013},  we   give the following  example of the Poisson random measure $N$ and $f$ satisfying (H.1)-(H.4) and (H.6):

Let $\{N(t)\}_{t\ge0}$ be a Poisson process with the rate $1$, $\{A^j\}_{ j\in\mathbb N}$   independent  and identically distributed random variables, with a common distribution function $F$, which are also independent of $\{N(t)\}_{ t\ge0}$. Then
  $$
N([0,t]\otimes B)=\sum_{j=1}^{N(t)}1_{B} (A^j), \ \ \ \ t\ge0, B\in \mathcal B(\mathbb R_+),
$$ 
is a   Poisson random measure  on the space $\mathbb R_+ \times \mathbb R_+$. The intensity measure of $\{N(t)\}_{ t\ge0}$ is given by 
$$\nu(A\otimes B)=\rho(A)\cdot F(B), \  \  \ A, B\in\mathcal B(\mathbb R_+).$$
Here  $\rho(\cdot)$ denotes the Lebesgue measure.

Assume that   there exists $a_0>0$ such that
$$
\int_0^{\infty} u^2 e^{a_0 u} F(du)<\infty. 
$$
For any function  $G\in C_b^1(\HH)$,  the function 
$$f(x,u):= G(x)u, \ \  \ \  \ x\in \HH, u\in \mathbb R_+$$  
   satisfies all the  conditions required in Theorem \ref{thm main}.

\end{rem}

\begin{remark} The   rate function $J$ can be expressed by the  entropy of Donsker-Varadhan, see \cite{DV}, \cite[Chapter V]{DS} or \cite[Section 2.2]{Wu01}.  Under  the Feller assumption:
 $$
 P_t(C_b(\HH))\subset C_b(\HH),  \ \ \ \forall t\ge0,
 $$
 we know  that (for instance  see Lemma B.7 in \cite{Wu00})
 \begin{align}\label{eq rate Wu00}
 J(\nu)=\sup\left\{-\int\frac{\mathfrak L \varphi}{\varphi}d  \nu; 1\le \varphi\in \mathbf D_e(\mathfrak L) \right\},  \ \ \nu\in \mathcal M_1(\HH).
 \end{align} 
\end{remark}

\begin{rem} For every $\varphi:\HH\rightarrow \RR$ measurable and bounded, as $\nu\mapsto\int_{\HH}\varphi d \nu$ is continuous w.r.t. the $\tau$-topology, then by the contraction principle (\cite{DS}), $$\mathbb P_{\nu}\left(\frac{1}{t}\int_0^t \varphi(X_s)ds\in \cdot\right)$$
satisfies the LDP on $\RR$ uniformly over $\nu$ in  $\mathcal M_{\lambda_0,L}$,  with the rate function given by
$$
J^{\varphi}(r):=\inf\left\{J(\beta)<+\infty; \beta\in\mathcal M_1(\HH)\  \text{and} \int \varphi d\beta=r \right\},\ \ \forall r\in\RR.
$$
\end{rem}

\begin{proof}[The proof of Theorem \ref{thm main}]
By Theorem \ref{Thm solu}, we know that $P_t$ is strong Feller and irreducible in $\HH$ for any $t>0$. According to   \cite[Theorem 2.1]{Wu01}, to prove Theorem  \ref{thm main},   it is sufficient to prove that for any $\lambda>0$, there exists a compact set  $\mathcal K$  in $\HH$
 \begin{equation}\label{lem condition 2} \ \
  \sup_{\nu\in\mathcal M_{\lambda_0, L}}\mathbb E_{\nu}\left[e^{\lambda\tau_{\mathcal K}}\right]<\infty, \ \  \ \text{and}  \ \ \ 
\sup_{x\in \mathcal K}\mathbb E_x\left[e^{\lambda\tau_{\mathcal K}^{(1)}}\right]<\infty, 
\end{equation}
where
\begin{equation}\label{stopping time}
\tau_{\mathcal K}:=\inf\{t\ge0; \   X_t\in \mathcal K\},\ \ \ \tau_{\mathcal K}^{(1)}:=\inf\{t\ge1;\  \ X_t\in \mathcal K\}.
\end{equation} 

 The basic ingredient for the proof of \eqref{lem condition 2} is to show the exponential decay of the tails of the stopping times $\tau_{\mathcal K}$ and  $\tau_{\mathcal K}^{(1)}$ for a suitable choice of compact set $\mathcal K\subset \HH$. It  can be proved by using   arguments in  \cite[Lemma 6.1]{Gou1} or \cite[Lemma 3.8]{WXX}, combining with the critical  exponential estimate in  Proposition \ref{Prop: exp} below.

The proof is complete.
\end{proof}

The following result is similar to Lemma 4.1 in \cite{RZ}.
\begin{lemma}\label{lem exp:mart} For any $g\in C_b^2(\HH)$,
$$
M_t^g:=\exp\left(g(X_t)-g(X_0)-\int_0^t h(X_s)ds \right)
$$
is an $\mathcal F_t$-local martingale, where
\begin{align}
h(x)=&\langle \Delta x, \nabla g(x)\rangle+\langle  B(x), \nabla g(x)\rangle+\frac12\left\|Q^{*}\nabla g(x)\right\|_{\HH}^2+\frac12\tr(Q^*{\Hess g}(x) Q)\\ \notag
&+\int_{\UU}\Big(\exp\big[g(x+f(x,u))-g(x)\big]-1-\big\langle \nabla g(x), f(x,u) \big \rangle  \Big)n(du).
\end{align}
\end{lemma}

\begin{proof} We follow the argument in \cite[Lemma 4.1]{RZ}.  Applying It\^o's formula first to $\exp(g(X_t))$ and then to $\exp\left(g(X_t)-g(X_0) -\int_0^t h(X_s)ds\right)$ proves the lemma.
\end{proof}

\begin{proposition}\label{Prop: exp} Assume that (H.1)-(H.4) and (H.6) hold. For any $\lambda \in(0,a_0], \theta\in(0,1)$, there exist   constants $c_1(\theta), c_2(\theta), c_3(\lambda)$ such that for any $T>0$,
\begin{align}\label{eq: exp V}
\mathbb E_x\left[\exp\left(\theta\lambda\int_0^T\|X_t\|_{\VV}dt\right)\right]\le  c_1(\theta)+c_2(\theta)e^{c_3(\lambda)T}e^{\lambda   \|x\|_{\HH}}.
\end{align}
\end{proposition}
\begin{remark}

 For the stochastic Burgers equation driven by Brownian motion (i.e., $f\equiv0$ in \eqref{eq:Burgers}),  the Lemma 5.2 of \cite{Gou1} tells us that for any $\lambda\in(0,\pi^2\|Q\|/2]$, where $\|Q\|$ is the norm of $Q$ as an operator in $\mathbb H$,
 \begin{align}\label{eq: exp V2}
\mathbb E_x\left[\exp\left(\lambda\int_0^T\|X_t\|_{\VV}^2dt\right)\right]\le  e^{\lambda {\mathrm tr}(Q) T} e^{\lambda   \|x\|_{\HH}^2}.
\end{align}
This  is  difficult to prove  in the jump case. Here, we replace  $\int_0^T\|X_t\|_{\VV}^2dt$  by $\int_0^T\|X_t\|_{\VV}dt$ and add an extra parameter $\theta\in(0,1)$ in \eqref{eq: exp V}, which is also enough  to show the exponential decay of the tails of the stopping times $\tau_{\mathcal K}$ and $\tau_{\mathcal K}^{(1)}$ for a suitable choice of $\mathcal K$. 
\end{remark}

\begin{proof}[The proof of Proposition \ref{Prop: exp}]
  For any $\lambda \in (0, a_0]$,
let $$Z_{\lambda}:=\int_0^T\frac{\lambda^2\|X_t\|_{\VV}^2}{\left(1+\lambda^2 \|X_t\|_{\HH}^2\right)^{\frac12}}dt.$$
Since $\|x\|_{\VV}\ge  \pi  \|x\|_{\HH}$, we have
\begin{align}\label{eq 4.9}
\frac{\lambda^2\|X_t\|_{\VV}^2}{\left(1+\lambda^2 \|X_t\|_{\HH}^2\right)^{\frac12}}
\ge   \left(1+\lambda^2 \|X_t\|_{\VV}^2\right)^{\frac12}-1.
\end{align}
Thus, to prove \eqref{eq: exp V}, it is enough to prove that for any $\lambda\in(0,a_0]$,  there exist   constants $c_1(\theta),  c_2(\theta), c_3(\lambda)$ such that
\begin{align}\label{eq: exp V1}
\mathbb E_x\left[\exp\left( \theta Z_{\lambda}   \right)\right]\le c_1(\theta) +c_2(\theta)e^{c_3(\lambda) T}e^{\lambda \|x\|_{\HH}}.
\end{align}
 
Let $$\psi_{\lambda}(x):=(1+\lambda^2\|x\|_{\HH}^2)^{\frac12},$$ a  generalization of  $\psi$ given in \eqref{eq: psi}. Then $\psi_{\lambda}$ has the similar estimates \eqref{eq: nabla psi}-\eqref{eq: B psi} with $\psi$ up to some constants.

Let  $G(x):=e^{\psi_{\lambda}(x)}$.
Note that
$$
{\Hess} G(x)=G(x){\Hess} \psi_{\lambda}(x)+ G(x)\nabla \psi_{\lambda}(x)\times \psi_{\lambda} (x).
$$
It implies that
\begin{equation}
\|\Hess G(x)\|_{\HS}\le \lambda^2 G(x), \ \ \forall x\in \VV.
\end{equation}
 By Taylor's expansion, there exist  constants $\theta_1\in(0,1), \theta_2\in(0,\theta_1)$ satisfying that
\begin{align}\label{eq: exp1}
&\left|\exp\big[ \psi_{\lambda}(x+f(x,u))- \psi_{\lambda}(x) \big]-1-\big\langle \nabla  \psi_{\lambda}(x), f(x,u)\big\rangle\right| \notag \\
=&\left| e^{- \psi_{\lambda}(x)}\big[ G(x+f(x,u))-G(x) -\langle \nabla G(x), f(x,u)\rangle\big]\right| \notag \\
=&\left|\frac{1}{2} e^{- \psi_{\lambda}(x)} \big\langle {\Hess}(G)(x+\theta_1 f(x,u))f(x,u), f(x,u)\big\rangle\right|\notag\\
\le& \frac{\lambda^2}{2}\exp\big( \psi_{\lambda}(x+\theta_1 f(x,u))- \psi_{\lambda}(x) \big)\|f(x,u)\|_{\HH}^2\notag\\
=& \frac{\lambda^2}{2}  \exp\big(\langle \nabla  \psi_{\lambda}(x+\theta_2 f(x,u)), \theta_1 f(x,u)\rangle \big)\|f(x,u)\|_{\HH}^2 \notag \\
\le & \frac{\lambda^2}{2} \exp\big(\lambda  \|f(x,u)\|_{\HH}\big)\|f(x,u)\|_{\HH}^2.
\end{align}
Applying Lemma \ref{lem exp:mart} with the above choice of $ \psi_{\lambda}$, we know that $$M^{ \psi_{\lambda}}(t):=\exp\left( \psi_{\lambda}(X_t)- \psi_{\lambda}(x)-\int_0^t h(X_s)ds\right)$$ is an $\mathcal F_t$-local martingale, where
\begin{align}\label{eq h1}
h(x):=&\langle  \Delta x, \nabla  \psi_{\lambda}(x)\rangle+\langle  B(x), \nabla  \psi_{\lambda}(x)\rangle+\frac12\left\|Q^*\nabla  \psi_{\lambda}(x)\right\|_{\HH}^2+\frac12\tr(Q^*{\Hess  \psi_{\lambda}}(x)Q) \notag \\
&+\int_{\UU}\left(\exp\left[ \psi_{\lambda}(x+f(x,u))- \psi_{\lambda}(x)\right]-1-\langle \nabla  \psi_{\lambda}(x), f(x,u) \rangle  \right)n(du) \notag \\
\le & -\frac{\lambda^2\|x\|_{\VV}^2}{(1+\lambda^2\|x\|_{\HH}^2)^{ \frac12}}+\lambda^2\|Q\|_{\HS}^2 +\frac{\lambda^2}{2}\int_{\UU}\exp\left(\lambda \|f(x,u)\|_{\HH} \right)\|f(x,u)\|_{\HH}^2 n(du).
\end{align}
By  (H.6), we know that for any fixed $\lambda\in(0, a_0]$,
\begin{equation}\label{eq h2}
M_{\lambda}:= \sup_{x\in \HH}\int_{\UU} \|f(x,u)\|_{\HH}^2 \exp\left(\lambda \|f(x,u)\|_{\HH} \right) n(du)<\infty.
\end{equation}
Hence, by \eqref{eq h1} and \eqref{eq h2},  we have that for any $r\ge0$,
\begin{align}\label{eq:Z}
&\mathbb P\left(Z_{\lambda}>r\right) \notag \\ 
\le &\mathbb P\left( \psi_{\lambda}(X_T)+  \int_0^T\frac{\lambda^2 \|X_s\|_{\VV}^2}{\left(1+\lambda^2\|X_s\|_{\HH}^2\right)^{\frac12}}ds>  r\right) \notag \\
= &\mathbb P\Bigg( \psi_{\lambda}(X_T)- \psi_{\lambda}(x)-\int_0^T h(X_s)ds+ \psi_{\lambda}(x)+\int_0^Th(X_s)ds \notag\\
 &\ \ \ \ + \int_0^T\frac{\lambda^2\|X_s\|_{\VV}^2}{\left(1+\lambda^2\|X_s\|_{\HH}^2\right)^{\frac12}}ds>   r\Bigg) \notag \\
\le  &\mathbb P\left( \psi_{\lambda}(X_T)- \psi_{\lambda}(x)-\int_0^T h(X_s)ds>   r - \psi_{\lambda}(x)-T\lambda^2\left(\frac{M_{\lambda}}{2}+ \|Q\|_{\HS}^2\right )\right) \notag \\
\le & \mathbb E\left[M^{ \psi_{\lambda}}_T \right] \exp\left(-r + \psi_{\lambda}(x)+T\lambda^2\left(\frac{M_{\lambda}}{2}+\|Q\|_{\HS}^2\right) \right) \notag \\
\le &  \exp\left(-r + \psi_{\lambda}(x)+T\lambda^2\left(\frac{M_{\lambda}}{2}+\|Q\|_{\HS}^2\right) \right),
\end{align}
where in the last  inequality we have used the fact that  $M_t^g$ is a non-negative local martingale.

For any $\theta\in(0,1)$ and $\lambda\le a_0$, by \eqref{eq:Z},  we have
\begin{align}
&\mathbb E\left[\exp \left( \theta Z_{\lambda} \right)\right] \notag \\
=& \theta+\theta\int_0^{\infty}   e^{ \theta r }   \mathbb P_x\left( Z_{\lambda}>r \right) dr \notag \\
\le & \theta+\theta\int_0^{\infty} e^{ \theta r }   \exp\left(-r + \psi_{\lambda}(x)+T\lambda^2\left(\frac{M_{\lambda}}{2}+\|Q\|_{\HS}^2\right)\right)dr \notag \\ 
=& \theta+ \exp\left( \psi_{\lambda}(x)+T\lambda^2\left(\frac{M_{\lambda}}{2}+\|Q\|_{\HS}^2\right)\right)  \frac{\theta}{1-\theta}.
\end{align}
This  implies \eqref{eq: exp V}.
The proof is complete.
\end{proof}

\noindent{\bf Acknowledgments}:  We sincerely thank the referee for helpful comments and remarks. S. Hu is supported by   the National Social Science (17BTJ034); R. Wang is supported by  the  NSFC(11871382), the Chinese State Scholarship Fund Award by the CSC   and the  Youth Talent Training Program by Wuhan University.

\end{document}